\documentclass{amsart}

\newtheorem{theorem}{Theorem}[section]
\newtheorem{lemma}[theorem]{Lemma}

\newtheorem{corollary}[theorem]{Corollary}

\theoremstyle{definition}

\theoremstyle{remark}

\numberwithin{equation}{section}

\newcommand{\R}{\mathbb{R}}

\newcommand{\N}{\mathbb{N}}

\newcommand{\x}{\mathfrak{x}}\newcommand{\y}{\mathfrak{y}}
\newcommand{\al}{\alpha}

\newcommand{\dX}{{\mathfrak{X}}}\newcommand{\dY}{{\mathfrak{Y}}}

\newcommand{\gor}{gr\dot{}\,}

\newcommand{\lie}{\mathfrak{g}}

\newcommand{\Fy}{\mathfrak{F}(\lie)}

\newcommand{\Lw}{L}

\newcommand{\Lh}{\mathcal{L}}
\newcommand{\Mo}{\mathcal{M}}

\begin{document}

\title{Stable flatness of nonarchimedean hyperenveloping algebras}
\author{Tobias Schmidt}
\address{Universit\"at M\"unster,Math. Institut, Einsteinstr.62, 48149 M\"unster, Germany}
\curraddr{} \email{toschmid@uni-muenster.de}
\thanks{}

\begin{abstract}
Let $L$ be a completely valued nonarchimedean field and $\lie$ a
finite dimensional Lie algebra over $L$. We show that its
hyperenveloping algebra $\Fy$ agrees with its Arens-Michael
envelope and, furthermore, is a stably flat completion of its
universal enveloping algebra. As an application, we prove that the
Taylor relative cohomology for the locally convex algebra $\Fy$ is
naturally isomorphic to the Lie algebra cohomology of $\lie$.
~\end{abstract}

 \maketitle
\section{Introduction}

Let $L$ be a $p$-adic local field and $G$ a locally $L$-analytic
group. A distinguished topological algebra appearing in the
locally analytic representation theory of $G$ is the
hyperenveloping algebra $\Fy$ associated to the Lie algebra $\lie$
of $G$ (\cite{T}). It equals a certain canonical locally convex
completion of the ordinary universal enveloping algebra $U(\lie)$,
similar to the well-known Arens-Michael envelope
(\cite{Helemski}). In the theory of complex Lie algebras it is
important to know when these completions are stably flat (or, in
other words, Taylor absolute localizations, cf. \cite{P}). Vaguely
speaking, a continuous ring extension between topological algebras
$\theta: A\rightarrow B$ is stably flat if the restriction functor
$\theta_*$ identifies the category of topological $B$-modules with
a full subcategory of topological $A$-modules in a way that leaves
certain homological relations invariant. We recall that such
extensions are of central importance in complex non-commutative
operator theory, partly because they leave the joint spectrum
invariant ( \cite{P},~\cite{T2}).

Turning back to our nonarchimedean setting it is therefore natural
to ask for nonarchimedean analogues of these complex results. In
this brief note we give a positive answer in full generality:
given a finite dimensional Lie algebra $\lie$ over an arbitrary
completely valued nonarchimedean field we define its
hyperenveloping algebra $\Fy$ and deduce that it always coincides
with the Arens-Michael envelope. Our main result is then that the
natural map \[U(\lie)\rightarrow\Fy\] is stably flat. We remark
that the first (second) result is in contrast to (in accordance
with) the complex situation (\cite{P}).

To give an application of our results recall that cohomology
theory for locally analytic $G$-representations (\cite{K2})
follows Taylor's general approach of a homology theory for
topological algebras (\cite{T1},\cite{T2}) where the locally
convex algebra of locally analytic distributions on $G$ takes up
the role of the base algebra. In view of the close relation
between this latter algebra and $\Fy$ (\cite{T}) it is to be
expected that locally analytic cohomology is closely related to
the cohomology relative to the base algebra $\Fy$. As a
consequence of our flatness result we find this latter cohomology
to be naturally isomorphic to the usual Lie algebra cohomology of
$\lie$.

\section{The result}
Throughout this note we freely use basic notions of nonarchimedean
functional analysis as presented in \cite{NFA}. Let $L$ be a
completely valued nonarchimedean field. We begin by recalling the
necessary relative homological algebra following
\cite{K2},\cite{P},\cite{T1}. We emphasize that as in
\cite{K2},\cite{T2} (but in contrast to \cite{P}) our preferred
choice of topological tensor product is the completed inductive
topological tensor product $\hat{\otimes}_\Lw$. By a {\it
topological algebra} $A$ we mean a complete Hausdorff locally
convex $\Lw$-space together with a separately continuous
multiplication. For a topological algebra $A$ we denote by
$\mathcal{M}_A$ the category of complete Hausdorff locally convex
$\Lw$-spaces with a separately continuous left $A$-module
structure (to ease notation we denote the right version by the
same symbol). Morphisms are continuous module maps and the
Hom-functor is denoted by $\Lh_A(.,.)$. A morphism is called {\it
strong} if it is strict with closed image and if both its kernel
and its image admit complements by closed $\Lw$-subspaces. The
category $\Mo_A$ is endowed with a structure of exact category by
declaring a sequence to be {\it $s$-exact} if it is exact as a
sequence of abstract vector spaces and if all occuring maps are
strong. Finally, a module $P\in\mathcal{M}_A$ is called {\it
$s$-projective} if the functor $\Lh_A(P,.)$ transforms short
$s$-exact sequences into exact sequences of abstract $\Lw$-vector
spaces. A {\it projective resolution} of $M\in\mathcal{M}_A$ is an
augmented complex $P_\bullet\rightarrow M$ which is $s$-exact and
where each $P_n$ is $s$-projective. A standard argument shows that
$\mathcal{M}_A$ has enough projectives and that any object admits
a projective resolution. As usual for a left resp. right
$A$-module $N$ resp. $M$ we denote by $M\hat{\otimes}_A N$ the
quotient of $M\hat{\otimes}_L N$ by the closure of the subspace
generated by elements of the form $ma\otimes n-m\otimes an, a\in
A, m\in M, n\in N$. Given a projective resolution
$P_\bullet\rightarrow M$ we define as usual
\[
\mathcal{T}or_*^{A}(M,N):=h_*(P_\bullet\hat{\otimes}_A N)~,~~~~~
\mathcal{E}xt^*_{A}(M,N):=h^*(\Lh_A(P_\bullet,N))
\]
 for $M,N\in\mathcal{M}_A$. These $L$-vector spaces do not depend
on the choice of $P_\bullet$ and have the usual functorial
properties.

Given a topological algebra $A$ so is the opposite algebra
$A^{op}$ and we may form the enveloping algebra
$A^{e}:=A\hat{\otimes}_\Lw A^{op}$ as a topological algebra. Given
a continuous ring homomorphism between topological algebras
$\theta: A\rightarrow B$ we may define a functor
$B^{e}\hat{\otimes}_{A^{e}}(.)$ from the category of $A$-bimodules
$\mathcal{M}_{A^{e}}$ to the category of $B$-bimodules
$\mathcal{M}_{B^{e}}$. The map $\theta$ is called {\it stably
flat} (or an {\it absolute localization}, cf. \cite{P},\cite{T2})
if the above functor transforms every projective resolution of
$A^{e}$ into a projective resolution of $B^{e}$.

Let $\lie$ be a finite dimensional Lie algebra over $L$ and let
$U(\lie)$ be its enveloping algebra. Denote by $M_\lie$ the
category of all (abstract) left $\lie$-modules. Fix a real number
$r>1$. Let $\x_1,...,\x_d$ be an ordered $L$-basis of $\lie$ with
$d={\rm dim}_L\lie$. Using the associated {\it
Poincar\'e-Birkhoff-Witt-basis} for $U(\lie)$ we define a vector
space norm on $U(\lie)$ via \begin{equation}\label{norms}
||\sum_\al d_\al\dX^\al||_{\dX,r}=\sup_\al
|d_\al|r^{|\al|}\end{equation} where
$\dX^\al:=\x_1^{\al_1}\cdot\cdot\cdot\x_d^{\al_d},~\al\in\N_0^d$.
We call the Hausdorff completion of $U(\lie)$ with respect to the
family of norms $||.||_{\dX,r},~r>1$ the {\it hyperenveloping
algebra} of $\lie$. Being a Hausdorff completion it comes equipped
with a natural map \[\theta: U(\lie)\rightarrow\Fy.\] Endowing the
source with the finest locally convex topology we will see below
that $\theta$ is a continuous homomorphism between well-defined
topological algebras in the above sense. In particular, we have
the categories $\mathcal{M}_{U(\lie)}$ and $\mathcal{M}_{\Fy}$ at
our disposal. Finally, the Hausdorff completion of $U(\lie)$ with
respect to {\it all} submultiplicative semi-norms on $U(\lie)$ is
called the {\it Arens-Michael envelope} $\hat{U}(\lie)$ of
$U(\lie)$ (cf. \cite{Helemski}, chap. V, \cite{P}, 6.1).
\begin{theorem}
The homomorphism $\theta$ is stably flat. It induces a natural
isomorphism of topological algebras
\[\hat{U}(\lie)\stackrel{\cong}{\longrightarrow}\Fy.\]
\end{theorem}
As with any stably flat homomorphism we obtain that the
restriction functor $\theta_*$ identifies $\mathcal{M}_{\Fy}$ with
a full subcategory of $\mathcal{M}_{U(\lie)}$ (cf. \cite{T2},
Prop. 1.2) leaving certain homological relations invariant
([loc.cit.], Prop.~1.4). Since in our setting $U(\lie)$ has the
finest locally convex topology one may go one step further and
pass to abstract Lie algebra cohomology.
\begin{corollary}
Given $M,N\in\mathcal{M}_{\Fy}$ the restriction functor $\theta_*$
induces natural vector space isomorphisms
\[\mathcal{T}{\rm or}^{\Fy}_*(M,N)\cong{\rm
Tor}^{U(\lie)}_*(M,N)~,~~~~~~\mathcal{E}{\rm
xt}_{\Fy}^*(M,N)\cong{\rm Ext}_{U(\lie)}^*(M,N).\]
\end{corollary}

\section{The proof}
\subsection{Norms on the hyperenveloping algebra}
We begin with two simple lemmas on the norm $||.||_{\dX,r},~r>1$.
Let $c_{ijk}\in L,~1\leq i,j,k\leq d$ denote the structure
constants of $\lie$ attached to the basis $\x_1,...,\x_d$.
\begin{lemma}\label{mult}
Suppose $|c_{ijk}|\leq 1$ for all $1\leq i,j,k\leq d$. Then
$||.||_{\dX,r}$ is multiplicative.
\end{lemma}
\begin{proof}
Note first that $L$ is a filtered ring (in the sense of
\cite{ST5}, sect. 1) via its absolute value and, by
multiplicativity, the associated graded ring is an integral
domain. Now put
$\dX^\al\,\dX^\beta=:\sum_\gamma\,c_{\al\beta,\gamma}\dX^\gamma$
with $c_{\al\beta,\gamma}\in L$. By hypothesis
$\dX^\al\dX^\beta=\x_1^{\al_1+\beta_1}\cdot\cdot\cdot\x_d^{\al_d+\beta_d}+\y$
where
$||\y||_{\dX,r}<||\x_1^{\al_1+\beta_1}\cdot\cdot\cdot\x_d^{\al_d+\beta_d}||_{\dX,r}$.
Hence $\sup_\gamma
|c_{\al\beta,\gamma}|r^{|\gamma|}=||\dX^\al\dX^\beta||_{\dX,r}=r^{|\al|+|\beta|}$
and therefore $|c_{\al\beta,\gamma}|\leq
r^{|\al|+|\beta|-|\gamma|}$ for all $\al,\beta,\gamma$. It follows
easily from this that $||.||_{\dX,r}$ is submultiplicative (cf.
\cite{ST5}, Prop. 4.2). Putting \[F^sU(\lie):=\{\lambda\in
U(\lie), ||\lambda||_{\dX,r}\leq p^{-s}\},\] $s\in\R$ turns
$U(\lie)$ into a filtered ring (cf. [loc.cit.], end of sect. 2).
Using the hypothesis again we obtain for $i<j$ that
\[||\x_i\x_j-\x_j\x_i||_{\dX,r}\leq r<r^2=||\x_i\x_j||_{\dX,r}.\]
The associated graded ring is therefore a polynomial ring over
$\gor L$ in the principal symbols $\sigma(\x_j)$. Thus, it is an
integral domain and therefore the norm on $U(\lie)$ must be
multiplicative.
\end{proof}
\begin{lemma}\label{independent}
The locally convex topology on $\Fy$ induced by the family
$||.||_{\dX,r},~r>1$ is independent of the choice of
$\x_1,...,\x_d$.
\end{lemma}
\begin{proof}
Since a homothety in $L^\times$ obviously induces an equivalent
locally convex topology we are reduced to show: given two bases
$\dX$ and $\dY$ of $\lie$ with integral structure constants there
is a real constant $C$ such that for any $\lambda\in U(\lie)$ and
any real $r>1$ we have $||\lambda||_{\dY,r}\leq
||\lambda||_{\dX,Cr}$. For all $i=1,...,d$ let
$\x_i=:\sum_{j=1,...,d} a_{ij}\y_j,~a_{ij}\in L$ and let
$C:=\max_{ij}(|a_{ij}|)$. If $\lambda=\sum_\al d_\al\dX^\al$ then
\[||\lambda||_{\dY,r}\leq\sup_\al\,|d_\al|\,(Cr)^{|\al|}=||\lambda||_{\dX,Cr}\]
using that $||.||_{\dY,r}$ is multiplicative by Lem. \ref{mult}.
\end{proof}
Recall that a {\it Hopf} $\hat{\otimes}$-{\it algebra} is a Hopf
algebra object in the braided monoidal category of topological
algebras. Without recalling more details on these definitions (cf.
\cite{P}, sect. 2) we state the
\begin{lemma}\label{Hopf} $\Fy$ is a
Hopf $\hat{\otimes}$-algebra with invertible antipode.
\end{lemma}
\begin{proof}
The locally convex vector space $\Fy$ is visibly a Fr\'echet space
(\cite{NFA}, \S8) whence the inductive and projective tensor
product topologies on $\Fy\otimes_L\Fy$ coincide ([loc.cit.],
Prop. 17.6). Fixing a basis $\x_1,...,\x_d$ of $\lie$ the topology
on the topological algebra $\Fy\hat{\otimes}_L\Fy$ may therefore
be described by the tensor product norms $||.||_{\dX,r}\otimes
||.||_{\dX,r}, r>1$. A direct computation now shows that the usual
Hopf algebra structure on $U(\lie)$ extends to the completion
$\Fy$.
\end{proof}
Remark: Suppose $L$ is a $p$-adic local field, $G$ is a locally
$L$-analytic group and $C^{an}_1(G,L)$ denotes the stalk at $1\in
G$ of germs of $L$-valued locally analytic functions on $G$. Then
$\Fy=C^{an}_1(G,L)'_b$ as locally convex vector spaces (\cite{K1},
Prop. 1.2.8) in accordance with the complex situation (cf.
\cite{R},\cite{P}). Analogous to this situation (cf. \cite{P},
sect. 8) the space $C^{an}_1(G,L)$ inherits a Hopf
$\hat{\otimes}$-algebra structure by functoriality in $G$. The
structure on $\Fy$ may then be obtained by passing to strong
duals.

\begin{corollary}\label{weak}
Suppose that the natural continuous map
\begin{equation}\label{augmentation}
\Fy\hat{\otimes}_{U(\lie)}L\longrightarrow L \end{equation} is a
topological isomorphism. If there exists a projective resolution
$P_\bullet\rightarrow L$ in $\mathcal{M}_{U(\lie)}$ such that the
base extension via $\theta$
\begin{equation}\label{base}\Fy\hat{\otimes}_{U(\lie)}P_\bullet\longrightarrow
\Fy\hat{\otimes}_{U(\lie)}L \end{equation} is a $s$-exact complex
in $\mathcal{M}_{\Fy}$, then $\theta$ is stably flat.
\end{corollary}
\begin{proof}
This follows from the above lemma together with \cite{P}, Prop.
3.7. Note that the proof of the latter proposition directly
carries over to our setting since it is based on formal properties
of Hopf $\hat{\otimes}$-algebras with invertible antipodes in
certain braided monoidal categories.
\end{proof}

Remark: A.Y. Pirkovski calls homomorphisms between complex
topological algebras satisfying ({\it mutatis mutandis}) the
conditions (\ref{augmentation}) and (\ref{base}) {\it weak
localizations}. The content of [loc.cit.], Prop. 3.7 is then that
a homomorphism between Hopf $\hat{\otimes}$-algebras with
invertible antipodes is a localization if and only if it is a weak
localization.

\bigskip
By Lem. \ref{mult} the map $\theta$ induces inclusions
\begin{equation}\label{incl}U(\lie)\subseteq\hat{U}(\lie)\subseteq\Fy.\end{equation}
On the other hand, suppose $||.||$ is a submultiplicative
semi-norm on $U(\lie)$. Choose a basis $\x_1,...,\x_d$ of $\lie$
with integral structure constants and let $r>1$ such that $r\geq
\max_j\,||\x_j||$. Given $\lambda\in U(\lie)$ write
$\lambda=\sum_{\al\in\N_0^d} d_\al\dX^\al,~d_\al\in L$ and compute
\[||\lambda||\leq\sup_\al |d_\al|~||\dX^\al||\leq\sup_\al |d_\al|
r^\al=||\lambda||_{\dX,r}.\] Hence, the second inclusion in
(\ref{incl}) is surjective and the second part of the theorem is
proved.

\subsection{Contracting homotopies}
We fix a real $r>1$, a $L$-basis $\x_1,...,\x_d$ of $\lie$ with
integral structure constants and let $||.||_r:=||.||_{\dX,r}$.
Recall the homological standard complex
$U_\bullet:=U(\lie)\otimes_L\bigwedge^\bullet\lie$ with
differential $\partial=\psi+\phi$ where
\[
\begin{array}{rl}
  \psi(\lambda\otimes\x_1\wedge...\wedge\x_q) &=\sum_{s<t}(-1)^{s+t}\lambda\otimes
[\x_s,\x_t]\wedge\x_1\wedge...\wedge
\widehat{\x_s}\wedge...\wedge\widehat{\x_t}\wedge...\wedge\x_q,
\\&\\
 \phi(\lambda\otimes\x_1\wedge...\wedge\x_q)&=
  \sum_{s}(-1)^{s+1}\lambda\x_s\otimes\x_1\wedge...\wedge
\widehat{\x_s}\wedge...\wedge\x_q\\
\end{array}
\]
(cf. \cite{CE}). Let $I_q$ be the collection of indices $1\leq
i_1<...<i_q\leq d$ and let $\lambda_q=\sum_{I\in I_q} u_I\otimes
x_I\in U_q$ with $u_I\in U(\lie), x_I=\x_{i_1}\wedge...\wedge
\x_{i_q}\in\bigwedge^q\lie$ be an element. If $\sum_q\lambda_q\in
U_\bullet$ we let \begin{equation}\label{rep}
||\sum_q\lambda_q||_r:=\sup_qr^q\sup_{I\in
I_q}\,||u_I||_r.\end{equation} By Lem. \ref{mult} $U_\bullet$
becomes in this way a faithfully normed left $U(\lie)$-module (in
the sense of \cite{BGR}, Def. 2.1.1/1).
\begin{lemma}\label{diff}
The differential $\partial$ is norm-decreasing on
$(U_\bullet,||.||_r)$.
\end{lemma}
\begin{proof}
Let $c_{ijk}$ denote the structure constants of $\x_1,...,\x_d$.
Since $|c_{ijk}|\leq 1$ and since $||.||_r$ is multiplicative on
$U(\lie)$ (Lem. \ref{mult}) we obtain
\[
\begin{split}
||\partial(\lambda_q)||_r\leq&\sup_{I\in
I_q}||\sum_{s<t}(-1)^{s+t}u_I\otimes
(\sum_kc_{stk}\x_k)\wedge\x_1\wedge...\wedge\widehat{\x_s}\wedge...\wedge\widehat{\x_t}\wedge...\wedge\x_q\\
& +\sum_{s}(-1)^{s+1}u_I\x_s\otimes\x_1\wedge...\wedge
\widehat{\x_s}\wedge...\wedge\x_q)||_r\\
\leq&\sup_{I\in
I_q}\,\max\,(\sup_{s<t}\,r^{q-1}||u_I||_r\,,\sup_s\,r^q||u_I||_r)\\
\leq & \sup_{I\in I_q}\,r^q ||u_I||_r\,=||\lambda_q||_r.
\end{split}
\]
\end{proof}
In the following we endow $U_\bullet$ with the locally convex
topology induced by the family of norms $||.||_r, r>1$. The last
result then implies that $\partial$ is continuous.

Recall that a {\it contracting homotopy} on an augmented
homological complex of $L$-vector spaces
$X_\bullet\stackrel{\epsilon}{\rightarrow} L$ is an $L$-linear map
$\eta: L\rightarrow X_0$ together with a family of $L$-linear maps
$s_q: X_q\rightarrow X_{q+1},~q\geq 0$ satisfying
\[
\begin{array}{ll}
  \epsilon\circ\eta(x)=x, & (\partial_{q+1}\circ s_q+s_{q-1}\circ\partial_{q})(y)=y, \\
   &  \\
s_0\circ \eta(x)=0,  & \partial_1\circ s_0(z)=z-(\eta\circ\epsilon)(z) \\
\end{array}
\]
for $x\in L,~y\in X_q,~z\in X_0,~q\geq 1$ (e.g. \cite{L},
V.1.1.4). Now suppose that $X_\bullet\rightarrow L$ is an
augmented complex in $\mathcal{M}_A$ for some topological
$L$-algebra $A$. If it admits a contracting homotopy such that the
maps $\eta$ and $s_q,~q\geq 0$ are continuous then it is $s$-exact
(\cite{T1}, remark after Def. 1.5).

~\\Recall that the augmented complex
$U_\bullet\stackrel{\epsilon}{\rightarrow} L$ has a distinguished
contracting homotopy $s$ ([loc.cit.], V.1.3.6.2). To review its
construction let $S(\lie)$ be the symmetric algebra of $\lie$. To
ease notation we denote its natural augmentation by $\epsilon$ as
well. Let
\[S_\bullet:=S(\lie)\otimes_L\dot{\bigwedge}\lie\stackrel{\epsilon}{\longrightarrow}L\]
be the augmented {\it Koszul complex} ([loc.cit.], V.1.3.3)
attached to the vector space $\lie$ with differential $\phi$ ({\it
mutatis mutandis}). Note that the choice of $\x_1,...,\x_d$
induces an isomorphism
\begin{equation}\label{iso} f:
U_\bullet\stackrel{\cong}{\longrightarrow}S_\bullet\end{equation}
as $L$-vector spaces compatible with $\epsilon$. The augmented
complex $S_\bullet\rightarrow L$ comes equipped with the following
contracting homotopy $\bar{s}$ depending on the basis
$\x_1,...,\x_d$ (cf. [loc.cit.], (1.3.3.4)). In case $d=1$ it is
given by the structure map $\eta: L\rightarrow S(\lie)$ together
with $\bar{s}_0:=S(\lie)\rightarrow S(\lie)\otimes\bigwedge^1\lie$
defined via $\bar{s}_0(\x_1^n)=\x_1^{n-1}\otimes \x_1$ for all
$n\in\N$ and $\bar{s}_0(1)=0$. In general, the definition is
extended to the tensor product
\[f^1_\bullet: S^1_\bullet\otimes_L\cdot\cdot\cdot\otimes_LS^d_\bullet\stackrel{\cong}{\longrightarrow}S_\bullet\]
by general principles (cf. [loc.cit.], V.1.3.2.). Here,
$S^j_\bullet$ equals the Koszul complex of the vector space
$L\x_j$ and $f^1_\bullet$ comes from functoriality of $S_\bullet$
applied to $\lie=\oplus_{j}L\x_j$. One obtains from $\bar{s}$ the
desired homotopy $s$ on $U_\bullet$ as follows: pulling $\bar{s}$
back to $U_\bullet$ via $f$ gives an $L$-linear map $\sigma$ on
$U_\bullet$.
It gives rise to maps \[\sigma^{(n)}: U_\bullet\longrightarrow
U_\bullet,~\sigma^{(n)}_q: U_q\longrightarrow U_{q+1}\] (cf.
[loc.cit.], Lem. V.1.3.5) having the property: for fixed $x\in
U_q$ the sequence $(\sigma^{(n)}_q(x))_{n\in\N}\subseteq U_{q+1}$
becomes eventually stationary ([loc.cit.], remark after formula
V.1.3.6.2). Then
\[s_q(x):=\lim_n\sigma^{(n)}_q(x)\] defines the desired contracting
homotopy $s$ on $U_\bullet\rightarrow L$.

Given $||.||_r, r>1$ on $U_\bullet$ we may use the basis
$\x_1,...,\x_d$ to define in complete analogy to (\ref{norms}) and
(\ref{rep}) norms on $S(\lie)$ and $S_\bullet$ which we also
denote by $||.||_r$.
\begin{lemma}\label{firststep}
The homotopy $\bar{s}$ is norm-decreasing on
$(S_\bullet,||.||_r)$.
\end{lemma}
\begin{proof}
By induction on $d={\rm dim}_L \lie$ we may endow the left-hand
side of the isomorphism $f^1_\bullet$ with the following norm:
\[||\lambda||_r:=\sup_{s+t=q}
\inf_{~(\lambda_s),(\mu_t)}||\lambda_s||_r\,||\mu_t||_r\] where
$\lambda\in (S^{i}_\bullet\otimes_L
S^j_\bullet)_q=\oplus_{s+t=q}S^{i}_s\otimes_L S^j_t$ is of the
form $\lambda=\sum_{s+t=q}(\sum\lambda_s\otimes\mu_t)$ and the
infimum is taken over all possible representations
$\sum\lambda_s\otimes\mu_t$ of the $(s,t)$-component of $\lambda$.
We claim that $f^1_\bullet$ is isometric. Again by induction we
are reduced to prove the claim for $f^2_q$ where
\[f^2_\bullet: S^{<d}_\bullet\otimes_L S^d_\bullet\stackrel{\cong}{\longrightarrow} S_\bullet\]
and $S^{<d}_\bullet$ equals the Koszul complex of
$\oplus_{j<d}L\x_j$. Fix $q\geq 0$. By definition of $||.||_r$ the
decomposition \[(S^{<d}_\bullet\otimes
S^d_\bullet)_q=\oplus_{s+t=q}S^{<d}_s\otimes S^d_t\] is
orthogonal. By definition of $f^2_q$ and since the elements
$\{1\otimes x_{I_q}\}_{I_q}$ are orthogonal in $S_q$, $f^2_q$
preserves this orthogonality in $S_q$. It therefore suffices to
fix $s+t=q$ and prove $||f^2_q(\lambda)||_r=||\lambda||_r$ for
$\lambda\in S^{<d}_s\otimes_L S^d_t$. In both cases ($s=q$ and
$s=q-1$) this is a straightforward computation whence
$f^1_\bullet$ is indeed isometric. Next we prove that $\bar{s}$ is
norm-decreasing on the left-hand side of the isomorphism
$f^1_\bullet$. For $d=1$ this follows since $\eta$ and $\bar{s}_0$
are certainly norm-decreasing. By induction we may suppose that
this is true on the complex $S^{<d}_\bullet$ and consider the
tensor product $S^{<d}_\bullet\otimes_LS^d_\bullet$. Let
$\lambda\in (S^{<d}_\bullet\otimes_L S^d_\bullet)_q$. Suppose
$q=0$ and hence $\lambda\in L$. It is then clear that
$||s(\lambda)||_r=||\eta(\lambda)\otimes 1||_r=||\lambda||_r$
where the first identity follows from formula [loc.cit.],
V.1.3.2.2. So assume $q>0$. Write
$\lambda=\sum_{s+t=q}(\sum\lambda_s\otimes\mu_t)$. Then
\[ \bar{s}(\lambda)=\sum_{s+t=q,s>0}\sum
\bar{s}(\lambda_s)\otimes\mu_t+\sum_{s+t=q,s=0}\sum
\bar{s}(\lambda_s)\otimes\mu_t+\eta\epsilon(\lambda_s)\otimes
\bar{s}(\mu_t).
\] according to the formulas [loc.cit.], V.1.3.2.2/1.3.2.3. Using the induction hypothesis on the right-hand side
one obtains $||\bar{s}(\lambda)||_r\leq ||\lambda||_r$ as desired.
\end{proof}

\begin{lemma}\label{cont} The homotopy $s$ is
continuous with respect to the locally convex topology on
$U_\bullet$.
\end{lemma}
\begin{proof}
Fix the norm $||.||_r,~r>1$ on $U_\bullet$ and $S_\bullet$. By
construction, the map $f$ appearing in (\ref{iso}) becomes
isometric. Hence, Lem. \ref{firststep} shows the $L$-linear map
$\sigma=f^{-1}\circ\bar{s}\circ f$ on $U_\bullet$ to be
norm-decreasing. The augmentation $\epsilon: U_0\rightarrow L$ and
the differential $\partial$ are also norm-decreasing (the latter
by Lem. \ref{diff}). Invoking the maps $\sigma^{(n)}$ from above
we deduce from $\sigma^{(0)}=\sigma$ and the formula
\[\sigma^{(n)}-\sigma^{(n-1)}=\sigma
(1-\epsilon-\partial\epsilon-\epsilon\partial)^n\] ([loc.cit.],
V.1.3.5.4) by induction that all $\sigma^{(n)}$ are
norm-decreasing. Now the contracting homotopy $s$ of $U_\bullet$
is defined as the pointwise limit
$s_q(x):=\lim_n\sigma_q^{(n)}(x),~x\in U_q$. Since the sequence
$\sigma^{(n)}_q(x)$ for $n\rightarrow\infty$ becomes eventually
stationary $s$ is seen to be norm-decreasing on $(U_\bullet,
||.||_r)$. Since $r>1$ was arbitrary the proof is complete.
\end{proof}
\subsection{Stable flatness}
We prove the remaining part of the theorem and the corollary of
section 2, respectively.
\begin{proof}
By Cor. \ref{weak} it suffices to show that the natural continuous
map
\begin{equation}\label{augmentation2}
\Fy\hat{\otimes}_{U(\lie)}L\longrightarrow L \end{equation}
 is a
topological isomorphism and secondly, that there exists a
projective resolution $P_\bullet\rightarrow L$ in
$\mathcal{M}_{U(\lie)}$ such that the base extension via $\theta$
\begin{equation}\label{base2}\Fy\hat{\otimes}_{U(\lie)}P_\bullet\longrightarrow
\Fy\hat{\otimes}_{U(\lie)}L \end{equation} is a $s$-exact complex
in $\mathcal{M}_{\Fy}$.

The augmentation ideal in $\Fy$ clearly generates the augmentation
ideal of $U(\lie)$. Hence, the map (\ref{augmentation2}) is a
bijection between finite dimensional $L$-vector spaces and
therefore topological. To prove the second condition we let
$P_\bullet:=U_\bullet$ together with the augmentation $\epsilon$.
The complex $U_\bullet$ consists of $s$-projective modules (which
are even {\it $s$-free} in the sense of \cite{P}, sect. 1) and,
admitting the (continuous) contracting homotopy $s$, the augmented
complex $P_\bullet\rightarrow L$ is $s$-exact. Using associativity
of $\hat{\otimes}$ the base extension (\ref{base2}) may be
identified with $\Fy\otimes_L\dot{\bigwedge}\lie\rightarrow L$ and
thus, equals the Hausdorff completion of the topologized complex
$U_\bullet\rightarrow L$. By continuity (Lem. \ref{cont}) the map
$s$ extends to this completion yielding a continuous contracting
homotopy on (\ref{base2}).
\end{proof}

\begin{proof}
Let $M,N\in\mathcal{M}_{\Fy}$. Any projective resolution
$P_\bullet\rightarrow M$ in ${\rm Mod}(\lie)$ by abstract free
$U(\lie)$-modules is a projective resolution in $\Mo_{U(\lie)}$
when endowed with the finest locally convex topology. One may
check that \cite{P}, Prop. 3.3 remains valid in our setting whence
the natural map $\Fy\hat{\otimes}_{U(\lie)}M\rightarrow M$ is an
isomorphism in $\Mo_{\Fy}$. Hence, by stable flatness
$\Fy\hat{\otimes}_{U(\lie)}P_\bullet\rightarrow M$ is a projective
resolution of $M$ in $\mathcal{M}_{\Fy}$. The claims follow now
from the isomorphisms of complexes
\[\Lh_{\Fy}(\Fy\hat{\otimes}_{U(\lie)}P_\bullet,N)\simeq
\Lh_{U(\lie)}(P_\bullet,N)\simeq {\rm Hom}_\lie(P_\bullet,N)\] and
\[N\hat{\otimes}_{\Fy} (\Fy\hat{\otimes}_{U(\lie)}P_\bullet)\simeq
N\hat{\otimes}_{U(\lie)}P_\bullet\simeq
N\otimes_{U(\lie)}P_\bullet\] where the last isomorphisms in both
rows follow from the fact that,in each degree, $P_\bullet$ carries
the finest locally convex topology.
\end{proof}
 Remark: Following \cite{T1}, Def. 2.3 one may also introduce
relative Hochschild cohomology in our nonarchimedean setting.
Another proof of the corollary (at least for the Tor-groups) may
then be deduced using analogues of \cite{T2}, Prop. 1.4 and
\cite{P}, Prop. 3.4.

\bibliographystyle{amsplain}

\end{document}